\documentclass[12pt]{article}
\usepackage{}
\usepackage[numbers,sort&compress]{natbib}

\usepackage{color}
\usepackage{threeparttable}

\usepackage[english]{babel}
\usepackage{t1enc}
\usepackage{amsthm,amsmath,amssymb}
\usepackage{mathrsfs}
\usepackage[mathscr]{eucal}
\usepackage[latin1]{inputenc}
\usepackage[english]{babel}
\usepackage{latexsym}
\usepackage{graphicx}
\usepackage[noend]{algpseudocode}
\usepackage{algorithmicx,algorithm}
\usepackage{lineno}
\usepackage{booktabs}
\usepackage{multirow}
\newtheorem{theorem}{Theorem}
\newtheorem{corollary}{Corollary}
\newtheorem{proposition}{Proposition}
\newtheorem{lemma}{Lemma}
\newtheorem{remark}{Remark}
\newtheorem{conjecture}{Conjecture}
\newtheorem{example}{Example}
\newtheorem{problem}{Problem}
\newtheorem{definition}{Definition}

\def\sfp{\textup{sfp}}
\def\sqrt{\textup{sqrt}}
\def \T{\textup{T}}

\def \rank{\textup{rank}}
\def \nullity{\textup{nullity}}

\def \diag{\textup{diag~}}

\def \Span{\textup{Span~}}

\makeatletter
\newcommand{\rmnum}[1]{\romannumeral #1}
\newcommand\restr[2]{{
		\left.\kern-\nulldelimiterspace 
		#1 
		\right|_{#2} 
}}
\newcommand{\Rmnum}[1]{\expandafter\@slowromancap\romannumeral #1@}
 
\makeatother
\title{An improved condition for a graph to be determined by its generalized spectrum}
\author{\small Wei Wang$^{{\rm a}}$\quad\quad Wei Wang$^{\rm b}$\thanks{Corresponding author: wang\_weiw@163.com}\quad\quad Fuhai Zhu$^{\rm c}$
\\
{\footnotesize$^{\rm a}$School of Mathematics, Physics and Finance, Anhui Polytechnic University, Wuhu 241000, P. R. China}\\
{\footnotesize$^{\rm b}$School of Mathematics and Statistics, Xi'an Jiaotong University, Xi'an 710049, P. R. China}\\
{\footnotesize$^{\rm c}$Departmant of Mathematics, Nanjing University, Nanjing 210093, P. R. China}
}
\date{}
\begin{document}
 \maketitle

\begin{abstract}
A fundamental and challenging problem in spectral graph theory is to characterize which graphs are uniquely determined by their spectra. In Wang [J. Combin. Theory, Ser. B, 122 (2017): 438-451], the author proved that an $n$-vertex graph $G$ is uniquely determined by its generalized spectrum (DGS) whenever $2^{-\lfloor\frac{n}{2}\rfloor}\det W$ is odd and square-free. Here, $W$ is the walk matrix of $G$, namely, $W=[e,Ae,\ldots,A^{n-1}e]$ with $e$ all-ones vector and $A$ the adjacency matrix of $G$.   In this paper, we focus on a larger family of graphs  with $d_n$ square-free, where $d_n$ refers to the last invariant factor of $W$. We introduce a new kind of polynomial for a graph $G$ associated with a prime $p$. Such a polynomial is invariant under generalized cospectrality. Using the newly defined polynomial, we obtain a sufficient condition for a graph in the larger family to be DGS. The main result of this paper  improves upon the  aforementioned  result of Wang while the proof for the main result gives a new way to attack the  problem of generalized spectral characterization of graphs. \\

\noindent\textbf{Keywords}: generalized spectrum; generalized spectral characterization; Smith normal form; square-free part

\noindent
\textbf{AMS Classification}: 05C50
\end{abstract}
\section{Introduction}
\label{intro}
Let $G$ be a simple graph with vertex set $\{1,2,\ldots,n\}$.  The \emph{adjacency matrix} of $G$ is the $n\times n$ symmetric matrix $A=(a_{i,j})$, where $a_{i,j}=1$ if $i$ and $j$ are adjacent;  $a_{i,j}=0$ otherwise. We often identify a graph $G$ with its adjacency matrix $A$. For example, the \emph{spectrum} of $G$, denoted by $\sigma(G)$, refers to the \emph{spectrum} of $A$, i.e., the roots (including multiplicities) of the characteristic polynomial
$\chi(A;x)=\det(x I-A)$ of $A$. Two graphs with the same spectrum are called \emph{cospectral}. Isomorphic graphs are clearly cospectral (as their adjacency matrices are  similar via a permutation matrix), but the converse is not true in general. A graph $G$ is \emph{determined by its spectrum} (DS for short) if any graph cospectral with $G$ is isomorphic to $G$. A fundamental and challenging problem in spectral graph theory is to  determine whether or not a given graph is DS. For basic results on spectral characterizations (determination)
of graphs, we refer the readers to the survey papers \cite{ervdamLAA2003,ervdamDM2009}.

The \emph{generalized spectrum} of a graph $G$ is the ordered pair $(\sigma(G),\sigma(\overline{G}))$, where $\overline{G}$ is the complement of $G$.  Naturally, two graphs are \emph{generalized cospectral} if they have the same generalized spectrum;  a graph $G$ is said to be \emph{determined by its generalized spectrum}
(DGS for short) if any graph generalized cospectral with $G$ is isomorphic to $G$. For a graph $G$, the \emph{walk matrix} of $G$ is
\begin{equation}
W=W(G):=[e,Ae,\ldots,A^{n-1}e],
\end{equation}
where $e$ is the all-ones vector. A graph $G$ is \emph{controllable} if $W(G)$ is nonsingular. We shall restrict ourselves to controllable graphs; the family of controllable graphs of order $n$ is denoted by $\mathcal{G}_n$.

The following simple arithmetic criterion for a controllable graph being DGS was proved in \cite{wang2013ElJC,wang2017JCTB}.
\begin{theorem}[\cite{wang2013ElJC,wang2017JCTB}]\label{sqf}
	Let $G\in \mathcal{G}_n$. If $2^{-\lfloor\frac{n}{2}\rfloor}\det W$ is odd and square-free, then $G$ is DGS.
\end{theorem}
Recently, Theorem \ref{sqf} has been extended or partially extended in various ways. For example, Qiu et al.~\cite{qiu2019DM} proved a similar result for the signless Laplacian spectrum. Li and Sun~\cite{li2021DM} considered the problem for $A_\alpha$-spectrum and unified Theorem \ref{sqf}  and the result of Qiu et al. \cite{qiu2019DM}. We refer to \cite{qiu2019EJC, qiu2021LAA,wang2020EJC,wang2021EUJC} for more results on the generalizations of Theorem \ref{sqf}.

The main aim of this paper is to improve upon Theorem \ref{sqf}, that is, to give a weaker  condition to guarantee a graph to be DGS. In general, if  $\det W$ contains a multiple odd prime factor then $G$ may not be DGS. To obtain a more effective sufficient condition, we use the notions of Smith normal forms and invariant factors of integral matrices. We briefly recall these notions with an additional assumption that the involved integral matrices are square and invertible.

Two $n\times n$ integral matrices $M_1$ and $M_2$ are \emph{integrally equivalent} if $M_2$ can be
obtained from $M_1$ by a sequence of the following operations: row permutation, row negation, addition of an integer multiple of one row to another and the corresponding column operations. Any integral invertible  matrix $M$ is integrally equivalent to a diagonal matrix $\diag[d_1,d_2 \ldots,d_n]$, known as the \emph{Smith normal form} of $M$, in which $d_1,d_2\ldots,d_n$ are positive integers with $d_i\mid d_{i+1}$ for $i = 1,2,...,n-1$. The diagonal elements $d_1,d_2\ldots,d_n$ are the \emph{invariant factors} of $M$. We  note that for an integral square matrix $M$, the determinant can be easily recovered,  up to a sign, from the Smith normal form. Indeed, $\det M=\pm d_1d_2\cdots d_n$. But it is generally impossible to determine the Smith normal form of $M$ from its determinant.

The following proposition obtained in \cite{wang2017JCTB} is an exception, which gives an equivalent description of the condition in Theorem \ref{sqf}.
\begin{proposition}[\cite{wang2017JCTB}]\label{maxrank}
	If	$\det W=\pm 2^{\lfloor\frac{n}{2}\rfloor}b$ for some odd and square-free integer $b$, then the Smith normal form of $W$ is
	$$\diag[\underbrace{1,1,\ldots,1}_{\lceil\frac{n}{2}\rceil},\underbrace{2,2,\ldots,2,2b}_{\lfloor\frac{n}{2}\rfloor}].$$
\end{proposition}
Now we introduce a  polynomial for a graph $G$ associated with a prime $p$, which plays a key role in this paper. We use $\mathbb{F}_p$ to denote the finite field of order $p$, and use $J$ to  denote the all-ones matrix (of order $n$).
\begin{definition}\label{Phip}\normalfont
	Let $p$ be an odd prime and $G$ be a graph with adjacency matrix $A$. We define	
	\begin{equation}\Phi_p(G;x)=\gcd(\chi(A;x),\chi(A+J;x))\in \mathbb{F}_p[x],
	\end{equation}
	where  the greatest common divisor (gcd) is taken over $\mathbb{F}_p$.
\end{definition}
\begin{remark}\label{invgc}\normalfont
	Write $f(t,x)=\chi(A+tJ;x)$, $t\in \mathbb{Z}$. Note that $f(t,x)$   is linear in $t$. It is not difficult to see that $\Phi_p(G;x)$ is invariant under generalized cospectrality. That is, if $G$ and $H$ are generalized cospectral, then $\Phi_p(G;x)=\Phi_p(H;x)$.
\end{remark}

Let $p$ be an odd prime and  $f\in \mathbb{F}_p[x]$ be a monic polynomial over the field $\mathbb{F}_p$. Now let $f = \prod_{1\le i\le r}f_i^{e_i}$
be the irreducible factorization of $f$, with distinct monic irreducible polynomials $f_1,f_2,\ldots, f_r$ and positive integers $e_1,e_2,\ldots, e_r$. The \emph{square-free part} of $f$, denoted by $\sfp (f)$,  is $\prod_{1\le i\le r}f_i$; see \cite[p.~394]{gathen}.

For an integral matrix $M$ and a prime $p$, we use $\rank_p M$ and $\nullity_p M$ to denote the rank and the nullity of $M$ over $\mathbb{F}_p$, respectively. We shall prove that for any graph $G$ and prime $p$,
\begin{equation}\label{basicupperbound}
\deg \sfp(\Phi_p(G;x))\le \nullity_p W(G).
\end{equation}
The main result of this paper is the following theorem.
\begin{theorem}\label{main}
	Let $G\in \mathcal{G}_n$ and $d_n$ be the last invariant factor of $W=W(G)$. Suppose that $d_n$ is square-free. If for each odd prime factor $p$ of $d_n$,
	\begin{equation}\label{keyequ}
	\deg\sfp(\Phi_p(G;x))= \nullity_p W,
	\end{equation} then $G$ is DGS.
\end{theorem}

We shall show that \eqref{keyequ} always holds for the case that $\nullity_p W=1$; see Corollary \ref{np1} in Section \ref{pt4}. Using Proposition \ref{maxrank}, we easily see that any graph satisfying the condition of Theorem \ref{sqf} necessarily satisfies the condition of Theorem \ref{main}. The converse is not true of course; as seen from later examples. This means that Theorem \ref{main} does improve upon Theorem \ref{sqf}. Furthermore, the proof of Theorem \ref{main}  gives an alternative proof of Theorem \ref{sqf}.

The main strategy in proving Theorem \ref{main} uses some ideas from \cite{qiuarXiv}. In \cite{qiuarXiv}, Qiu et al. strengthen Theorem \ref{sqf} in a different way. The argument developed in  \cite{qiuarXiv} gives a new proof of Theorem \ref{sqf}.  Nevertheless, their argument essentially depends on the assumption that $ \nullity_p W=1$. To overcome this restriction, we generalize a familiar property for the characteristic polynomial of a symmetric matrix over $\mathbb{R}$ to the case of $\mathbb{F}_p$ or its extension. This is the main aim of Section 2. The proof of Theorem \ref{main} is given in Section 3. Some examples and discussions are given in the last section.

\section{Orthogonality over an extension field of $\mathbb{F}_p$}
Throughout this section, we assume that $p$ is a fixed odd prime. Let $\overline{\mathbb{F}}_p$ be the algebraic closure of the finite field  $\mathbb{F}_p$. Let  $\overline{\mathbb{F}}_p^n$  denote the linear space consisting of all $n$-dimensional column vectors over $\overline{\mathbb{F}}_p$. Two vectors $u,v\in \overline{\mathbb{F}}_p^n$ are called \emph{orthogonal} if $u^\T v=0$. The notation for this is $u \perp v$. Naturally, two subspaces $U$ and $V$ are called \emph{orthogonal} and denoted by $U\perp V$, if $\xi\perp \eta$  for any $\xi\in U$ and $\eta \in V$.
\begin{definition}[\cite{babai1992}]\normalfont
For  a subspace  $V$  of $\overline{\mathbb{F}}_p^n$,  the \emph{orthogonal space} of $V$ is
	\begin{equation}
	V^\perp=\{u\in \overline{\mathbb{F}}_p^n\colon\, v^\T u=0\text{~for every~$v\in V$} \}.
	\end{equation}
	\end{definition}
Of course, 	$V^\perp$ is a subspace of $\overline{\mathbb{F}}_p^n$ and $V^\perp$ has dimension $n-\dim V$. A major difficulty here is that $V^\perp\cap V$ may contain some nonzero vector and hence   $\overline{\mathbb{F}}_p^n=V\oplus V^\perp$ does \emph{not} hold in general. This explains why we do not call $V^\perp$  the  \emph{orthogonal complement} of $V$, a name usually used in Euclidian space $\mathbb{R}^n$. A subspace $V\subset \mathbb{F}^n$ is \emph{isotropic} if $V\cap V^\perp$ contains a nonzero vector. Otherwise it is \emph{anisotropic} \cite{babai1992}. Note that  $(\overline{\mathbb{F}}_p^n)^\perp$ contains only zero vector and hence $\overline{\mathbb{F}}_p^n$ is anisotropic by definition.
\begin{lemma} [{{\cite[p.270]{roman}}}]\label{equforani}Let $U$ and $V$ be two subspace of $\overline{\mathbb{F}}_p^n$ with $U\subset V$. Then
	\begin{equation}\label{inequdim}
	\dim (U^\perp \cap V)\ge \dim V -\dim U.
	\end{equation}
	Moreover, the equality in \eqref{inequdim} holds if $V$ is anisotropic.
\end{lemma}
\begin{proof}
	Note that $\dim U^\perp=n-\dim U$. We have
	\begin{equation}\label{upv}
	\dim (U^\perp \cap V)=(n-\dim U) +\dim V-\dim (U^\perp +V).
	\end{equation}
Thus, \eqref{inequdim}	holds as $\dim (U^\perp +V)\le n$.
Now suppose that $V$ is anisotropic. By definition, we have $V^\perp\cap V=\{0\}$ and hence $\dim (V^\perp +V)=\dim V^\perp+\dim V=n$. Noting that $V^\perp+V\subset U^\T +V\subset \overline{\mathbb{F}}_p^n$ as $U\subset V$, we must have  $\dim (U^\T +V)=n$. By \eqref{upv}, the equality in \eqref{inequdim} holds.
\end{proof}
Let $A$ be an $n\times n$ matrix over $\overline{\mathbb{F}}_p$. We usually identify $A$ as a linear transformation (also denoted by $A$) on $\overline{\mathbb{F}}_p^n$ defined by $A\colon\,x\mapsto Ax$.  A subspace $U\subset \overline{\mathbb{F}}_p^n$ is \emph{$A$-invariant} if $AU\subset U$, that is, if  $Ax\in U$ for any $x\in U$. For an $A$-invariant subspace $U$, we use $\restr{A}{U}$ to denote the linear transformation $A$ restricted to $U$.
\begin{lemma}\label{facchi}
	If $A$ is a symmetric matrix over $\overline{\mathbb{F}}_p$ and $U$ is an $A$-invariant subspace of $\overline{\mathbb{F}}_p^n$. Then $U^\perp$ is $A$-invariant and
	\begin{equation}\label{ff}
	\chi(A;x)=\chi(\restr{A}{U};x)\chi(\restr{A}{U^\perp};x).
	\end{equation}
\end{lemma}
\begin{proof}
	The first assertion is simple as one can check that the usual argument for the same assertion in the  field $\mathbb{R}$ is also valid for $\overline{\mathbb{F}}_p$. Nevertheless, we need some extra work to establish (\ref{ff}) as the equality $\overline{\mathbb{F}}_p^n=U\oplus U^\perp$ may fail.
	
	Let $\chi(A;x)=(x-\lambda_1)^{v_1}\cdots(x-\lambda_k)^{v_k}$, where $\lambda_1,\ldots,\lambda_k$  are distinct roots of $\chi(A;x)$. Let $V_i=\mathcal{N}(A-\lambda_i I)^{v_i}$ be the nullspace of $(A-\lambda_i I)^{v_i}$. Then by the primary decomposition theorem (see e.g.~\cite{hoffman1971}), we have
	
	\noindent(\rmnum{1}) each $V_i$ is $A$-invariant;
	
	\noindent(\rmnum{2}) $\dim V_i=v_i$ and $\chi(\restr{A}{V_i};x)=(x-\lambda_i)^{v_i}$;
	
	\noindent(\rmnum{3}) $\overline{\mathbb{F}}_p^n=V_1\oplus\cdots\oplus V_k$;
	
	\noindent(\rmnum{4}) there are polynomials $h_1,\ldots,h_k$ such that each $h_i(A)$ is
	the identity on $V_i$ and is zero on all the other $V_i$'s.
	
	Noting that $U$ is $A$-invariant, we have
		\begin{equation}\label{Uop}
	U=(U\cap V_1)\oplus\cdots\oplus (U\cap V_k),
	\end{equation}
	see \cite[p.~264]{hoffman1971}. Similarly, as $U^\perp$ is also $A$-invariant, we have
		\begin{equation}\label{Uperp}
	U^\perp =(U^\perp \cap V_1)\oplus\cdots\oplus (U^\perp \cap V_k).
	\end{equation}

\noindent\emph{Claim} 1: $V_i\perp V_j$ for all distinct $i$ and $j$.

 Let $\xi$ and $\eta$ be any vectors in $V_i$ and $V_j$ respectively. As $h_i(A)$ is the identity on $V_i$ and is zero on $V_j$ , we have $h_i(A)\xi=\xi$ and  $h_i(A)\eta=0$. Noting that $A^\T =A$, we have
 \begin{equation}
 \xi^\T \eta=(h_i(A)\xi)^\T \eta=\xi^\T (h_i(A))^\T \eta=\xi^\T (h_i(A) \eta)=0.
 \end{equation}
This proves Claim 1.

\noindent\emph{Claim} 2: Each $V_i$ is anisotropic.

Let $V_i'=\oplus_{j\neq i}V_j$. By (\rmnum{3}), we see that  $\dim V_i'=n-\dim V_i$. On the other hand, by Claim 1, we know that $V_i\perp V_j$ for $j\neq i$ and hence  $V_i\perp V_i'$, i.e., $V_i'\subset V_i^\perp$.  Noting that   $\dim V_i^\perp=n-\dim V_i$, the two spaces $V_i'$ and $V_i^\perp$ must coincide. Therefore,  $V_i\cap V_i^\perp=V_i\cap V_i'=\{0\}$ and Claim 2 follows.	

\noindent\emph{Claim} 3:  $U^\perp \cap V_i=(U\cap V_i)^\perp \cap V_i$ for each $i$.

Let  $U_i=U\cap V_i$ for $i\in \{1,\ldots,k\}$. As $U_i\subset U$, we have $U_i^\perp\supset U^\perp$ and hence $U_i^\perp \cap V_i \supset U^\perp\cap V_i$. It remains to show that $U_i^\perp \cap V_i \subset U^\perp\cap V_i$. Pick any $\xi\in U_i^\perp \cap V_i$. As $\xi\in V_i$, Claim 1 implies that $\xi\perp V_j$ and hence $\xi\perp U_j$ for any $j\neq i$. This, together with the fact that $\xi\in U_i^\perp$, implies that $\xi\perp U_j$ for all $j\in\{1,\ldots,k\}$. Noting that $U=U_1\oplus\cdots\oplus U_k$ by \eqref{Uop}, we have $\xi\perp U$, i.e., $\xi\in U^\perp$. Thus, $\xi\in U^\perp \cap V_i$ and hence $U_i^\perp \cap V_i \subset U^\perp\cap V_i$ by the arbitrariness of $\xi$. This proves Claim 3.

By Claim 3, we can rewrite \eqref{Uperp} as
\begin{equation}\label{Uperpo}
U^\perp =(U_i^\perp \cap V_1)\oplus\cdots\oplus(U_k^\perp\cap V_k).
\end{equation}
Let $u_i=\dim U_i$, and $w_i=\dim (U_i^\perp \cap V_i)$ for $i\in\{1,\ldots,k\}$. Note that $U_i\subset V_i$, $\dim V_i=v_i$,  and $V_i$ is anisotropic by Claim 2. It follows from  Lemma \ref{equforani} that $\dim (U_i^\perp\cap V_i)=\dim V_i-\dim U_i$, i.e.,
\begin{equation}\label{wvu}
w_i=v_i-u_i.
\end{equation}  Note that $U_i$ is $A$-invariant and $U_i\subset V_i$. We see that $\chi(\restr{A}{U_i};x)$ is a factor of $\chi(\restr{A}{V_i};x)$ and hence $\chi(\restr{A}{U_i};x)=(x-\lambda_i)^{u_i}$. Consequently, we have $\chi(\restr{A}{U};x)=(x-\lambda_1)^{u_1}\cdots(x-\lambda_k)^{u_k}$. Similarly, by \eqref{Uperpo}, we have
$\chi(\restr{A}{U^\perp};x)=(x-\lambda_1)^{w_1}\cdots(x-\lambda_k)^{w_k}$. Thus, \eqref{ff} holds by \eqref{wvu}. This completes the proof.
\end{proof}
\section{Proof of Theorem \ref{main}}\label{pt4}

An orthogonal matrix $Q$ is called \emph{regular} if $Qe=e$ (or equivalently, $Q^\T e=e$). An old result of Johnson and Newman \cite{johnson1980JCTB} states that two graphs $G$ and $H$ are generalized cospectral if and only if there exists a regular orthogonal matrix $Q$ such that $Q^\T A(G)Q=A(H)$. For controllable graphs, the corresponding matrix $Q$ is unique and rational.
\begin{lemma}[\cite{johnson1980JCTB,wang2006EuJC}]\label{gcQ}
	Let $G\in \mathcal{G}_n$ and $H$ be a graph generalized cospectral with $G$. Then there exists a unique regular rational orthogonal matrix $Q$ such that $Q^\T A(G) Q=A(H)$. Moreover, the unique $Q$ satisfies $Q^\T =W(H)W^{-1}(G)$ and hence is rational.
\end{lemma}
For a controllable graph $G$, define $\mathcal{Q}(G)$ to be the set of all regular rational orthogonal matrices  $Q$ such that $Q^\T A(G)Q$ is an adjacency matrix. For a rational matrix $Q$, the \emph{level} of $Q$, denoted by $\ell(Q)$, or simply $\ell$, is the smallest positive integer $k$ such that $kQ$ is an integral matrix. Note that a  regular rational orthogonal matrix with level one is a permutation matrix.
The following two important results are direct consequences of Lemma \ref{gcQ}.
\begin{lemma}[\cite{wang2006EuJC}]\label{pdn}
		Let $G\in \mathcal{G}_n$ and $d_n$ be the last invariant factor of $W$. Then $\ell(Q)\mid d_n$ for any $Q\in \mathcal{Q}(G)$.
\end{lemma}
\begin{lemma}[\cite{wang2006EuJC}]\label{onelevel}
	Let $G\in \mathcal{G}_n$. Then $G$ is DGS if and only if $\ell(Q)=1$ for each $Q\in \mathcal{Q}(G)$.
\end{lemma}
\begin{lemma}[\cite{wang2006EuJC}]\label{tdw}
	For any graph $G$ of order $n$, we have $2^{\lfloor\frac{n}{2}\rfloor}\mid\det W$.
\end{lemma}
  For nonzero integers $d$, $m$ and positive integer $k$, we use $d^k\mid\mid m$ to indicate that $d^k$ precisely divides $m$, i.e., $d^k\mid m$ but $d^{k+1}\nmid m$.  The following result was obtained in \cite{wang2017JCTB} using an involved argument; we refer to \cite{qiuarXiv} for a simpler proof.
\begin{lemma}[\cite{wang2017JCTB,qiuarXiv}]\label{oddl}
	Let $G\in \mathcal{G}_n$. If $2^{\lfloor\frac{n}{2}\rfloor}\mid\mid \det W$  then any $Q\in\mathcal{Q}(G)$ has  odd level.
\end{lemma}
\begin{lemma}[\cite{wang2021LAA}]\label{tmf}
	For any graph $G$ of order $n$,  at most $\lfloor\frac{n}{2}\rfloor$ invariant factors of W are congruent to $2$ modulo $4$.
\end{lemma}
\begin{corollary}\label{oddlevel}
	Let $G\in \mathcal{G}_n$ and $d_n$ be the last invariant factor of $W$. If $d_n\equiv 2\pmod{4}$  then any $Q\in \mathcal{Q}(G)$ has odd level.
	\end{corollary}
\begin{proof}
Since $d_n\equiv 2\pmod{4}$ and $d_1\mid d_2\mid \cdots\mid d_n$, each invariant factor is either odd or congruent to $2$ modulo $4$. It follows from Lemma \ref{tmf} that $2^{\lfloor\frac{n}{2}\rfloor+1}\nmid\det W$. By Lemma \ref{tdw}, we see that $2^{\lfloor\frac{n}{2}\rfloor}\mid\mid\det W$. The assertion follows by Lemma \ref{oddl}.	
\end{proof}
The remaining part of this section is devoted to showing that, for any $Q\in \mathcal{Q}(G)$ with  $G$  satisfying the condition of Theorem \ref{main}, the level $\ell(Q)$ contains none odd prime factor.  We begin with a fundamental property on the column vectors of $W$.
\begin{lemma}[\cite{liesen2012,qiu2019DM}]\label{firstr}
	Let $r=\rank_p W$. Then the first $r$ columns of $W$ are linearly independent over $\overline{\mathbb{F}}_p$ 	and hence constitute a basis of the column space of $W$.
\end{lemma}
\begin{definition}\normalfont
Let $p$ be an odd prime.	The \emph{$p$-main polynomial} of a graph $G$, denoted by $m_p(G;x)$, is  the monic polynomial $f\in \mathbb{F}_p[x]$ of
	smallest degree such that $f(A)e = 0$.
\end{definition}
We recall that the ordinary \emph{main polynomial} $m(G;x)$ (over $\mathbb{Q}$) can be defined in the same manner; see \cite{teranishi2002LMA,rowlinson2007AADM}. It is known that the ordinary main polynomial is invariant under generalized cospectrality. Unfortunately, the $p$-main polynomial does not have such a nice property in general. In other words, two generalized cospectral graphs $G$ and $H$ may have different $p$-main polynomials for some odd prime $p$; see Remark \ref{dispmain} in Section \ref{dissec}. However, a key intermediate result of this paper shows that such an inconsistency can never happen under the restriction that one graph, say $G$, satisfies the assumption of Theorem \ref{main}. The overall idea is simple. We shall show that under the condition of Theorem \ref{main}, there is a direct connection between the $p$-main polynomial $m_p(G;x)$ and the polynomial $\Phi_p(G;x)$ which is invariant under generalized cospectrality (see Eq. \eqref{formp} in Lemma \ref{basicinq}).

To simplify the notations in the following lemmas, we fix a graph $G$ and use $A$ and $W$ to denote the adjacency matrix and walk matrix of $G$, respectively.
\begin{definition}
	$A_t=A+tJ$ and $W_t=[e,A_t e,\ldots,A_t^{n-1}e]$ for $t\in \overline{\mathbb{F}}_p$.
\end{definition}
\begin{lemma}\label{invt}
	$\mathcal{N}(W_t^\T)$ is constant on $t\in \overline{\mathbb{F}}_p$.
\end{lemma}
\begin{proof}
Note that $J\xi=(ee^\T)\xi=(e^\T \xi)e\in \Span\{e\}$ for any $\xi\in  \overline{\mathbb{F}}_p^n$. Thus,  for any $t\in  \overline{\mathbb{F}}_p$ and positive integer $k$,  there exist $c_0,\ldots,c_{k-1}\in  \overline{\mathbb{F}}_p$ such that
\begin{equation}
(A+tJ)^k e=A^k e+\sum_{i=0}^{k-1}c_iA^i e.
\end{equation}
It follows that there exists an $n\times n$ upper triangular matrix $U$ with 1 on the diagonal such that
\begin{equation}
[e,(A+tJ)e,\ldots,(A+tJ)^{n-1}e]=[e,Ae,\ldots,A^{n-1}e]U,
\end{equation}
i.e., $W_t=W U$. Thus, $W_t^\T=U^\T W^\T$ and hence  $\mathcal{N}(W_t^\T)=\mathcal{N}(W^\T)$ as $U^\T$ is invertible.
\end{proof}

\begin{lemma}\label{invs}
$\mathcal{N}(W^\T)$ is an  $(A+tJ)$-invariant subspace for any $t\in\overline{\mathbb{F}}_p$.
\end{lemma}
\begin{proof}
Let $\chi(A;x)=c_0+c_1x+\cdots+c_{n-1} x^{n-1}+x^{n}$ and $C$ be the companion matrix, that is,
\begin{equation}
C=\begin{pmatrix}
0&0&\cdots&0&-c_0\\
1&0&\cdots&0&-c_1\\
0&1&\cdots&0&-c_2\\
\vdots&\vdots&\ddots&\vdots&\vdots\\
0&0&\cdots&1&-c_{n-1}
\end{pmatrix}.
\end{equation}
It follows from the Cayley-Hamilton Theorem that $A^ne=-c_0e-c_1Ae-\cdots-c_{n-1}A^{n-1}e$ and hence $AW=WC$, or equivalently, $W^\T A=C^\T W^\T$ as $A$ is symmetric. Let $\xi$ be any vector in $\mathcal{N}(W^\T)$. Then we have $W^\T (A\xi)=C^\T W^\T \xi=0$ and hence $A\xi\in \mathcal{N}(W^\T)$. Moreover, as $e^\T$ is the first row of $W^\T$, we see that $e^\T \xi=0$ and hence $J \xi=0$. Thus,
$(A+tJ)\xi=A\xi\in \mathcal{N}(W^\T)$. This indicates that $\mathcal{N}(W^\T)$ is $(A+tJ)$-invariant, as desired.
\end{proof}

\begin{lemma}\label{mainchi} 	$
m_p(G;x)=\chi(\restr{A}{\mathcal{N}^\perp(W^\T)};x).
	$
\end{lemma}
\begin{proof}
	Let $r=\rank_p W$ and $f=\chi(\restr{A}{\mathcal{N}^\perp(W^\T)};x)$. Then $\deg f=\dim \mathcal{N}^\perp(W^\T)=r$. By Lemma \ref{firstr}, we see that $A^ke\in \Span\{e,Ae,\ldots,A^{k-1}e\}$ if and only if $k\ge r$. This implies that $\deg m_p(G;x)=r$. Thus, it suffices to show $f(A)e=0$. Indeed, by Cayley-Hamilton Theorem, we have $\restr{f(A)}{\mathcal{N}^\perp(W^\T)}$ is zero.  As $e\perp \xi$ for any $\xi\in \mathcal{N}(W^\T)$, we see that $e\in \mathcal{N}^\perp(W^\T)$. Therefore, $f(A)e=0$ and we are done.
\end{proof}
\begin{lemma}\label{sameroots}
	$\chi(\restr{A}{\mathcal{N}(W^\T)};x)$ divides $\Phi_p(G;x)$, and $\sfp (\Phi_p(G;x))$ divides $\chi(\restr{A}{\mathcal{N}(W^\T)};x)$.
\end{lemma}
\begin{proof}
By Lemma \ref{invs},  the space $\mathcal{N}(W^\T)$ is $(A+tJ)$-invariant for any $t\in \overline{\mathbb{F}}_p$.  Let $f_t\in \overline{\mathbb{F}}_p[x]$ denote $\chi(\restr{(A+tJ)}{\mathcal{N}(W^\T)};x)$. Since $\restr{J}{\mathcal{N}(W^\T)}$ is zero, we find that $f_t$ does not depend on $t$. Clearly $f_t\mid \chi(A+tJ;x)$. Since $f_0=f_1$, we have  $f_0\mid \gcd(\chi(A;x),\chi(A+J;x))$, which is exactly the first assertion.

To prove the second assertion, it suffices to show that every root of $\Phi_p(G;x)$ is a root of $f_0$ (or $f_1$). Let $\lambda\in \overline{\mathbb{F}}_p$ be any root of $\Phi_p(G;x)$, that is, $\lambda$ is a common eigenvalue of $A$ and $A+J$. Then there exist two nonzero vectors $\xi$ and $\eta$ such that $A\xi=\lambda \xi$ and $(A+J)\eta=\lambda \eta$. We claim that  either $e^\T \xi=0$ or $e^\T \eta=0$.  Actually, we have
\begin{equation}
\xi^\T (\lambda I-A)\eta=\xi^\T J\eta=\xi^\T ee^\T\eta=(e^\T\xi)(e^\T \eta).
\end{equation}
Taking transpose and noting that  $A$ is symmetric, we  have $\xi^\T (\lambda I-A)\eta=\eta^\T (\lambda I-A)\xi=0$. Thus $(e^\T\xi)(e^\T \eta)=0$ and the claim follows. Suppose that $e^\T\xi=0$. Then $e^\T A^k\xi=e^\T \lambda^k \xi=0$ for any positive $k$ and hence $W^\T \xi=0$, i.e., $\xi\in \mathcal{N}(W^\T)$. Since $\xi$ is an eigenvector of $\restr{A}{\mathcal{N}(W^\T)}$, the corresponding eigenvalue $\lambda$ must be a root of $f_0$. Now suppose that  $e^\T\eta=0$. Similarly we have $\eta\in \mathcal{N}(W_1^\T)$. But $\mathcal{N}(W_1^\T)=\mathcal{N}(W^\T)$ by Lemma \ref{invt}. Thus, $\eta\in \mathcal{N}(W^\T)$ and we see that $\lambda$ must be a root of $f_1$. Recall that $f_0=f_1$. We find that $\lambda$ is always a root of $f_0$. This completes the proof.
\end{proof}
\begin{lemma}\label{basicinq}
$\deg\sfp(\Phi_p(G;x))\le \nullity_p W\le \deg\Phi_p(G;x)$. Moreover, if the first equality holds then
	\begin{equation}\label{formp}
	m_p(G;x)=\frac{\chi(A;x)}{\sfp(\Phi_p(G;x))}.
	\end{equation}
\end{lemma}
\begin{proof}
	Note that $\deg\chi(\restr{A}{\mathcal{N}(W^\T)};x)=\dim \mathcal{N}(W^\T)=\nullity_p W$. The first assertion clearly follows from Lemma \ref{sameroots}.  Note that $\deg m_p(G;x)=\rank_p W=n-\nullity_p W$. It follows from Lemmas \ref{mainchi}, \ref{facchi} and \ref{sameroots} that
		\begin{eqnarray}\label{twoinq}
		n-\nullity_p W &=&\deg m_p(G;x)\nonumber\\
		&=&\deg \chi(\restr{A}{\mathcal{N}^\perp(W^\T)};x)\nonumber\\ \nonumber
		&= &\deg\frac{\chi(A;x)}{\chi(\restr{A}{\mathcal{N}(W^\T)};x)}\nonumber\\
		&\le&\deg\frac{\chi(A;x)}{\sfp(\Phi_p(G;x))}\\
		&=&n-\deg\sfp(\Phi_p(G;x)).\nonumber
	\end{eqnarray}
Suppose that  $\deg\sfp(\Phi_p(G;x))= \nullity_p W$.  Then the inequality in \eqref{twoinq}  must become an equality. Clearly, this happens precisely when $ \chi(\restr{A}{\mathcal{N}(W^\T)};x)=\sfp(\Phi_p(G;x))$. Thus, \eqref{formp} holds and the proof is complete.
\end{proof}
\begin{corollary}\label{np1}
If $\nullity_p W=1$ then $\deg\sfp(\Phi_p(G;x))=1$.
\end{corollary}
\begin{proof}
	As  $\nullity_p W=1$, Lemma \ref{basicinq} implies that $\deg\sfp(\Phi_p(G;x))\le 1$ and $\deg\Phi_p(G;x)\ge 1$. Now clearly, $\Phi_p(G;x)$ has the form $(x-\lambda)^k$ for some $\lambda\in \overline{\mathbb{F}}_p$ (indeed $\lambda\in \mathbb{F}_p$) and positive integer $k$. Thus, $\sfp(\Phi_p(G;x))=x-\lambda$ and the corollary follows.
\end{proof}
\begin{corollary}\label{samepmain}
Let $G\in \mathcal{G}_n$ and $d_n$ be the last invariant factor of $W(G)$. Suppose that $d_n$ is square-free and $p$ is an odd prime factor of $d_n$. If $
\deg\sfp(\Phi_p(G;x))= \nullity_p W(G)$,
then $\nullity_p W(G)=\nullity_p W(H)$ and  $m_p(G;x)=m_p(H;x)$  for any graph $H$ generalized cospectral with $G$.	
\end{corollary}
\begin{proof}
 Write $k=\nullity_p W(G)$. Then exactly the last $k$ invariant factors $d_{n-k+1},\ldots,d_n$ of $W(G)$ are multiple of $p$. Since $p\mid\mid d_n$ and $d_{n-k+1}\mid d_{n-k+2}\mid\cdots\mid d_n$, all these invariant factors must have $p$ as a \emph{simple} factor. Thus $p^k\mid\mid\det W(G)$ and hence $p^k\mid\mid\det W(H)$ as $\det W(G)=\pm \det W(H)$. Consequently, we have $\nullity_p W(H)\le k$.  On the other hand, noting that $\Phi_p(G;x)=\Phi_p(H;x)$,   Lemma \ref{basicinq} together with the condition of this proposition implies
		$$\nullity_p W(H)\ge \deg\sfp(\Phi_p(H;x))=\deg\sfp(\Phi_p(G;x))= \nullity_p W(G)=k.$$
Therefore, we have  $\nullity_p W(H)=k$.		Now, using the second part of Lemma \ref{basicinq} for both $G$ and $H$, we find that $m_p(G;x)=m_p(H;x)$.
\end{proof}
The following corollary is not needed for the proof of Theorem \ref{main} but will be used to give a better understanding of the counterexample given in the next section.
\begin{corollary}\label{samepmain2}
Let $G\in \mathcal{G}_n$ and $d_n$ be the last invariant factor of $W(G)$. Suppose that $d_n$ is square-free and $p$ is an odd prime factor of $d_n$. If $\nullity_p W(G)=2$ then, for any graph $H$ generalized cospectral with $G$, one of the following two statements holds.\\
(\rmnum{1})  $\nullity_p W(H)=2$ and $m_p(G;x)=m_p(H;x)$;\\
(\rmnum{2)}  $\nullity_p W(H)=1$ and $m_p(G;x)\neq m_p(H;x)$.
\end{corollary}
\begin{proof}
	By Lemma \ref{basicinq}, we have $\deg\sfp(\Phi_p(G;x))\le 2\le \deg\Phi_p(G;x)$. Thus, we have  $\deg\sfp(\Phi_p(G;x))=2$ or $1$. If $\deg\sfp(\Phi_p(G;x))=2$, then (\rmnum{1}) holds by Corollary \ref{samepmain}. Now assume that $\deg\sfp(\Phi_p(G;x))=1$. Then, using a similar argument as in the proof of Corollary \ref{samepmain}, we have $p^2\mid\mid \det W(H)$ and hence $\nullity_p W(H)=1$ or $2$. If $\nullity_p W(H)=1$ then the two polynomials $m_p(G;x)$ and $m_p(H;x)$ have different degrees and of course $m_p(G;x)\neq m_p(H;x)$. It remains to consider the case that $\nullity_p W(H)=2$.
	
	Since $\deg\sfp(\Phi_p(G;x))=1$ and $\deg \Phi_p(G;x)\ge 2$, we have $\Phi_p(G;x)=(x-\lambda)^k$ for some $\lambda\in\mathbb{F}_p$ and integer $k\ge 2$. By Lemma \ref{sameroots}, we see that $\chi(\restr{A}{\mathcal{N}(W^\T(G))};x)$ is a factor of  $\Phi_p(G;x)$. As $\deg \chi(\restr{A}{\mathcal{N}(W^\T(G))};x)=\nullity_p W(G)=2$, we must have $\chi(\restr{A}{\mathcal{N}(W^\T(G))};x)=(x-\lambda)^2$. Thus, by Lemmas \ref{mainchi} and \ref{facchi}, we have
	$m_p(G;x) =\frac{\chi(A(G);x)}{(x-\lambda)^2}$. Since $\nullity_p W(H)=2$, the same argument also works for $H$. Noting that $\chi(A(H);x)=\chi(A(G);x)$ and $\Phi_p(H;x)=\Phi_p(G;x)$, we see that 	$m_p(G;x)=m_p(H;x)$. This completes the proof.
\end{proof}
\begin{proposition}\label{almostmain}
Let $Q\in\mathcal{Q}(G)$ with level $\ell$. If $p\mid\mid d_n$ and	$\deg\sfp(\Phi_p(G;x))=\nullity_p W$ then $p\nmid\ell$.
\end{proposition}
	\begin{proof}
Let $A=A(G)$ and $A'=Q^\T A Q$.  Let $f(x)\in \mathbb{Z}[x]$ be a monic polynomial such that $f(x)\equiv m_p(G;x)\pmod p$. By Corollary \ref{samepmain}, we have $f(A)e\equiv f(A')e\equiv 0\pmod{p}$. Write $k=\nullity_p W$. Note that $\deg f(x)=n-k$.  Define
\begin{equation}	\overline{W}=\left[e,Ae,\ldots,A^{n-k-1}e,\frac{1}{p}f(A)e,\frac{1}{p}Af(A)e,\ldots,\frac{1}{p}A^{k-1}f(A)e\right]
\end{equation}
and
\begin{equation}	\overline{W'}=\left[e,A'e,\ldots,A'^{n-k-1}e,\frac{1}{p}f(A')e,\frac{1}{p}A'f(A')e,\ldots,\frac{1}{p}A'^{k-1}f(A')e\right].
\end{equation}
Then both $\overline{W}$ and  $\overline{W'}$ are integral matrices and we still have $Q^\T \overline{W} =\overline{W'}$. This indicates that $\ell(Q^\T)\mid \det \overline{W}$, or equivalently,  $\ell\mid \det \overline{W}$. On the other hand,  as $p^{k}\mid\mid \det W$ and  $\det \overline{W}=p^{-k}\det W$, we see that  $p\nmid \det \overline{W}$.  Thus, $p\nmid \ell$, as desired.
\end{proof}
Now, we are in a position to present the proof of Theorem \ref{main}.
\begin{proof}[Proof of Theorem \ref{main}]The case that $n=1$ is trivial and hence we assume $n\ge 2$.  Let $Q$ be any matrix in $\mathcal{Q}(G)$ and $\ell$ be its level.  Noting that $n\ge 2$, Lemma \ref{tdw} implies that $\det W$ and hence $d_n$ is even. Since $d_n$ is square-free, we see that $d_n\equiv 2\pmod{4}$. It follows from Corollary \ref{oddlevel} that $\ell$ is odd. In order to show $\ell=1$, we need to show that $\ell$ has no odd prime factor. Suppose to the contrary that there is an odd prime $p$ such that $p\mid \ell$. By Lemma \ref{pdn}, we know that $\ell\mid d_n$ and hence $p\mid d_n$. Moreover, as $d_n$ is square-free, we must have $p\mid\mid d_n$. Now, by Proposition \ref{almostmain}, we have $p\nmid \ell$. This is a contradiction. Therefore, $\ell=1$ and $G$ is DGS by Lemma \ref{onelevel}. This completes the proof.	
	\end{proof}
\section{Discussions}\label{dissec}
We first give an example  to illustrate that Theorem \ref{main} does improve upon Theorem \ref{sqf}. We use Mathematica for the  computation.
\begin{example}\normalfont
	Let $n=16$ and $G$ be the graph with adjacency matrix
$$ A=\scriptsize{\left(
\begin{array}{cccccccccccccccc}
0 & 0 & 1 & 0 & 0 & 1 & 0 & 1 & 1 & 0 & 1 & 0 & 0 & 0 & 0 & 0 \\
0 & 0 & 0 & 1 & 0 & 1 & 1 & 1 & 1 & 0 & 1 & 0 & 0 & 1 & 0 & 1 \\
1 & 0 & 0 & 1 & 1 & 0 & 1 & 1 & 1 & 0 & 0 & 0 & 0 & 1 & 0 & 1 \\
0 & 1 & 1 & 0 & 1 & 0 & 1 & 0 & 0 & 0 & 1 & 0 & 0 & 1 & 0 & 1 \\
0 & 0 & 1 & 1 & 0 & 0 & 0 & 1 & 1 & 0 & 0 & 1 & 1 & 1 & 1 & 1 \\
1 & 1 & 0 & 0 & 0 & 0 & 1 & 1 & 1 & 1 & 0 & 0 & 0 & 0 & 1 & 1 \\
0 & 1 & 1 & 1 & 0 & 1 & 0 & 0 & 1 & 0 & 1 & 0 & 0 & 1 & 1 & 1 \\
1 & 1 & 1 & 0 & 1 & 1 & 0 & 0 & 1 & 1 & 0 & 1 & 1 & 0 & 0 & 1 \\
1 & 1 & 1 & 0 & 1 & 1 & 1 & 1 & 0 & 0 & 1 & 0 & 0 & 1 & 1 & 1 \\
0 & 0 & 0 & 0 & 0 & 1 & 0 & 1 & 0 & 0 & 0 & 1 & 0 & 1 & 1 & 1 \\
1 & 1 & 0 & 1 & 0 & 0 & 1 & 0 & 1 & 0 & 0 & 0 & 0 & 0 & 0 & 1 \\
0 & 0 & 0 & 0 & 1 & 0 & 0 & 1 & 0 & 1 & 0 & 0 & 1 & 1 & 1 & 0 \\
0 & 0 & 0 & 0 & 1 & 0 & 0 & 1 & 0 & 0 & 0 & 1 & 0 & 0 & 0 & 0 \\
0 & 1 & 1 & 1 & 1 & 0 & 1 & 0 & 1 & 1 & 0 & 1 & 0 & 0 & 1 & 1 \\
0 & 0 & 0 & 0 & 1 & 1 & 1 & 0 & 1 & 1 & 0 & 1 & 0 & 1 & 0 & 1 \\
0 & 1 & 1 & 1 & 1 & 1 & 1 & 1 & 1 & 1 & 1 & 0 & 0 & 1 & 1 & 0 \\
\end{array}
\right)}.
$$
The Smith normal form of $W(G)$ is
	$$\diag[\underbrace{1,1,1,1,1,1,1,1}_{8},\underbrace{2,2,2,2,2,2,2\times3,2b}_{8}],$$
	where $b=3\times 23\times 29\times 1225550789\times6442787651$, which is square-free. From the Smith normal form, we see that Theorem \ref{sqf} is not applicable here. We turn to Theorem \ref{main}. Consider $p=3$. Then $\Phi_p(G;x)=x^4+2 x^3+2 x^2+x+1$, which has the standard factorization $\Phi_p(G;x)=\left(x^2+x+2\right)^2$ over $\mathbb{F}_p$. Thus,
	$\sfp~\Phi_p(G;x)=x^2+x+2$. As $\nullity_p W= 2$, we see that \eqref{keyequ} holds in this case. Moreover, by Corollary \ref{np1}, all other odd prime factors of $b$ (or $2b$) must satisfy $\eqref{keyequ}$. Thus $G$ is DGS by Theorem \ref{main}.
\end{example}
Our next example indicates that if \eqref{keyequ} is not satisfied, then $G$ may not be DGS.
\begin{example}\label{ex2}\normalfont
Let $n=9$ and $G$ be the graph with adjacency matrix

$$ A=\scriptsize{\left(
	\begin{array}{ccccccccc}
	0 & 1 & 0 & 1 & 0 & 0 & 1 & 1 & 1 \\
	1 & 0 & 1 & 0 & 1 & 0 & 0 & 1 & 1 \\
	0 & 1 & 0 & 1 & 1 & 1 & 0 & 1 & 1 \\
	1 & 0 & 1 & 0 & 1 & 0 & 0 & 0 & 0 \\
	0 & 1 & 1 & 1 & 0 & 1 & 1 & 1 & 0 \\
	0 & 0 & 1 & 0 & 1 & 0 & 1 & 0 & 0 \\
	1 & 0 & 0 & 0 & 1 & 1 & 0 & 1 & 1 \\
	1 & 1 & 1 & 0 & 1 & 0 & 1 & 0 & 1 \\
	1 & 1 & 1 & 0 & 0 & 0 & 1 & 1 & 0 \\
	\end{array}
	\right)}.
$$
The Smith normal form of $W$ is
$$\diag[1,1,1,1,1,2,2,2\times 3\times 5,2\times 3\times 5].$$
Now we see $\nullity_3 W=\nullity_5 W=2$.
Direct computation (using Mathematica) indicates that
$\sfp (\Phi_3(G;x))=x+2$ (over $\mathbb{F}_3$) and $\sfp (\Phi_5(G;x))=x^2+x+1$ (over $\mathbb{F}_5$). Thus, \eqref{keyequ} holds for $p=5$ but not for $p=3$.  This means that for this graph,  Proposition \ref{almostmain} is applicable only for $p=5$. Therefore, we can not eliminate the possible that there exists some $Q\in \mathcal{Q}(G)$ with level $3$. Indeed, such a $Q$ does exist for this particular example. Let
$$Q=\scriptsize{\frac{1}{3}\left(
\begin{array}{ccccccccc}
1 & -1 & 0 & 2 & 1 & 0 & -1 & 1 & 0 \\
-1 & 1 & 0 & 1 & 2 & 0 & 1 & -1 & 0 \\
1 & -1 & 0 & -1 & 1 & 0 & 2 & 1 & 0 \\
0 & 0 & 0 & 0 & 0 & 0 & 0 & 0 & 3 \\
1 & 2 & 0 & -1 & 1 & 0 & -1 & 1 & 0 \\
-1 & 1 & 0 & 1 & -1 & 0 & 1 & 2 & 0 \\
0 & 0 & 3 & 0 & 0 & 0 & 0 & 0 & 0 \\
2 & 1 & 0 & 1 & -1 & 0 & 1 & -1 & 0 \\
0 & 0 & 0 & 0 & 0 & 3 & 0 & 0 & 0 \\
\end{array}
\right).}$$
Then $Q^\T AQ$ is an adjacency matrix of a graph. This indicates that $G$ is not DGS by Lemma \ref{onelevel}.
\begin{remark}\label{dispmain}\normalfont
	Let $H$ be the graph with adjacency matrix $Q^\T AQ$, where $A$ and $Q$ are matrices as described in Example \ref{ex2}.
We claim that  $m_p(G;x)\neq m_p(H;x)$ for $p=3$. Otherwise, noting that $\deg m_3(G;x)=2$ and  using the same procedure as in the proof of Proposition \ref{almostmain}, we would get that $3\nmid\ell(Q)$, which is a contradiction. Actually, $m_3(G;x)=x^7+2 x^6+2 x^5+x^4+2 x^3+2 x^2+x$ and $m_3(H;x)=x^8+x^7+2 x^5+x^4+2 x^2+2 x$.
\end{remark}
\begin{remark}\normalfont
	Let $G$ and $H$ be a pair of generalized cospectral graphs whose walk matrices have the same Smith normal form as follows:
	$$\diag[\underbrace{1,\ldots,1}_{\lceil\frac{n}{2}\rceil},\underbrace{2,\ldots,2,2b_1,2b_2}_{\lfloor\frac{n}{2}\rfloor}],$$
	where $b_2$ (and hence $b_1$) is odd and square-free. We claim that $G$ and $H$ must be isomorphic. Let $Q$ be the regular rational orthogonal matrix such that $Q^\T A(G)Q=A(H)$. We need to eliminate the possibility that $p\mid \ell(G)$ for any odd prime factor $p$ of $b_1$. Note that for such a prime $p$, Corollary \ref{samepmain2} clearly implies that $m_p(G)=m_p(H)$. Consequently, using the same argument as in the proof of Proposition \ref{almostmain}, we can show that $\ell(Q)\mid p^{-2}\det W(G)$. This means $p\nmid \ell(Q)$, as desired.
\end{remark}
We end the discussion of Example \ref{ex2} by suggesting the following natural and interesting problem.
\begin{problem}
Let $G$ and $H$ be a pair of generalized cospectral graphs whose walk matrices have the same Smith normal form as follows:
$$\diag[\underbrace{1,\ldots,1}_{\lceil\frac{n}{2}\rceil},\underbrace{2,\ldots,2,2b_1,2b_2,\ldots,2b_k}_{\lfloor\frac{n}{2}\rfloor}],$$
where $b_k$ (and hence each $b_i$) is odd and square-free. Suppose that $k\ge 3$. Can we still guarantee that $G$ and $H$ are isomorphic?	
\end{problem}

\end{example}
To see the extent to which  Theorem \ref{main} improves upon Theorem \ref{sqf}, we have performed a series of numerical experiments.  The graphs are randomly generated using the random graph model $G(n;p)$ model with $p=1/2$. For each $n\in\{10,15,\ldots,50\}$  we generated 1,000 graphs randomly, and
counted the number of graphs satisfying the condition of Theorem \ref{sqf} and Theorem \ref{main}, respectively. To see how often that \eqref{keyequ} is met under the assumption that $d_n$ is square-free, we also record the number of graphs satisfying this assumption. Table 1 records one of such experiments. For example, for $n=10$, among 1,000 graphs generated in one experiment, $261$ graphs have a square-free invariant factor $d_n$. For these $261$ graphs, $226$ graphs satisfy the condition of Theorem \ref{sqf} while $253$ graphs satisfy the condition of Theorem \ref{main}. The remaining $8$  graphs  do not satisfy \eqref{keyequ} and hence we do not know whether they are DGS or not.
\begin{table}[htbp]
	\footnotesize
	\centering
	\caption{\label{computer} Comparison between Theorem \ref{sqf} and Theorem \ref{main}}
	\begin{threeparttable}
	\begin{tabular}{ccccc}
		\toprule
		$n$ &\# graphs & \#DGS & \#DGS  &\#Unknown\\
			(graph order) & (with $d_n$ square-free\tnote{*} ) &(by Theorem \ref{sqf})  & (by Theorem \ref{main}) &(by Theorem \ref{main})\\
		\midrule
		10  &  261 &226&253&8\\
		15  &  283 &217&265&18\\
		20  & 268 &228&262&6\\
		25  & 254 &221&245&9\\
		30  &  257 &213&243&14\\
		35  &  252 &204&245&7\\
		40  &  280 &238&270&10\\
		45  &  250 &204&237&13\\
		50  &  275 &224&259&16\\
			\bottomrule
	\end{tabular}
\begin{tablenotes}
	\footnotesize
	\item[*]The numbers $d_n$ are usually  huge integers and hence complete  factorizations are unavailable in a reasonable time.  We use the fast command FactorInteger[$d_n$,\textbf{Automatic}] in Mathematica to factor $d_n$. Note that this command extracts only factors that are easy to find.
\end{tablenotes}
\end{threeparttable}
\end{table}

At the end of this paper, we would like to suggest a possible improvement on Theorem \ref{main}. We begin with a definition.
\begin{definition}\normalfont
Let $f\in \mathbb{F}_p[x]$ be a monic polynomial with irreducible factorization $f =\prod_{1\le i\le r}f_i^{e_i}$.
	  We define the \emph{square-root} of $f$, denoted by $\sqrt (f)$,  to be $\prod_{1\le i\le r}f_i^{\lceil\frac{e_i}{2}\rceil}$.
\end{definition}
We remind the reader that $(\sqrt (f))^2\neq f$ unless all $e_i$'s  are even. Note that $\sqrt (f)$ is always a multiple of  $\sfp (f)$, and they are equal precisely when all $e_i$ are either one or two. Thus, for any graph $G$ and prime $p$, we always have
\begin{equation}
\deg \sfp (\Phi_p(G;x))\le \deg \sqrt(\Phi_p(G;x)).
\end{equation}
While Lemma \ref{sameroots} tells us $\sfp(\Phi_p(G;x))$ divides $\chi(\restr{A}{\mathcal{N}(W^\T)};x)$, it seems that the corresponding result also holds if we replace  $\sfp(\Phi_p(G;x))$ by  $\sqrt (\Phi_p(G;x))$. If we can show this improvement,  then we can strengthen Inequality \eqref{basicupperbound} as
\begin{equation}
\deg \sqrt(\Phi_p(G;x))\le \nullity_p W(G),
\end{equation}
and moreover we can improve upon Theorem \ref{main} simply by replacing $\sfp(\Phi_p(G;x))$ with $\sqrt(\Phi_p(G;x))$. We write such a possible improvement on Theorem \ref{main} as the following conjecture.
\begin{conjecture}
\label{mainconj}
		Let $G\in \mathcal{G}_n$ and $d_n$ be the last invariant factor of $W=W(G)$. Suppose that $d_n$ is square-free. If for each odd prime factor $p$ of $d_n$,
		\begin{equation}\label{keyequconj}
		\deg\sqrt(\Phi_p(G;x))=\nullity_p W,
		\end{equation} then $G$ is DGS.
\end{conjecture}
\section*{Acknowledgments}
This work is supported by the	National Natural Science Foundation of China (Grant Nos. 12001006, 11971376 and 11971406) and the Scientific Research Foundation of Anhui Polytechnic University (Grant No.\,2019YQQ024).

\end{document}